\documentclass[11pt]{article}       % onecolumn (second format)
\usepackage{graphicx}
\usepackage{mathptmx}      % use Times fonts if available on your TeX system
\usepackage{mathtools}
\usepackage{times}
\usepackage{mathtools}
\usepackage{amsmath,amssymb,amsthm}

\usepackage{bbm}
\usepackage{algorithmic}
\usepackage{algorithm}
\usepackage{color}
\usepackage[title]{appendix}
\newcommand{\Rb}{\mathbbm{R}}      % for Real numbers

\newcommand{\Eb}{\mathbbm{E}}
\newcommand{\Fc}{\mathcal{F}}

\newcommand{\Sc}{\mathcal{S}}

\newcommand{\argmin}{\mathop{\rm argmin}}

\newcommand{\dist}{\mathop{\rm dist}}

\newcommand*{\dt}[1]{\overset{\hbox{\,\tiny${}_\bullet$}}{#1}}

\newcommand{\eqdef}{\mathrel{\overset{\raisebox{-0.02em}{$\scriptstyle\vartriangle$}}{\,=\,}}}

\newtheorem{proposition}{Proposition}[section]
\newtheorem{theorem}[proposition]{Theorem}
\newtheorem{corollary}[proposition]{Corollary}
\newtheorem{lemma}[proposition]{Lemma}
\newtheorem{definition}[proposition]{Definition}
\begin{document}

\title{Convergence of a Stochastic Subgradient Method with Averaging for Nonsmooth Nonconvex Constrained
Optimization\footnote{This publication was supported by the NSF Award DMS-1312016. }
}


%\titlerunning{Stochastic Subgradient Method with Averaging}        % if too long for running head

\author{Andrzej Ruszczy\'nski\footnote{Rutgers University, Department of Management Science and Information Systems, Piscataway, NJ 08854, USA;
              email: {rusz{@}rutgers.edu}}
}


\date{December 16, 2019}

\maketitle

\begin{abstract}
We prove convergence of a single time-scale stochastic subgradient method with subgradient averaging
for constrained problems with a nonsmooth and nonconvex objective function having the property of
generalized differentiability. As a tool of our analysis, we also
prove a chain rule on a path for such functions.\\
\emph{Keywords:} {Stochastic Subgradient Method, Nonsmooth Optimization, Generalized Differentiable Functions, Chain Rule}
%\subclass{90C15 \and 90C48}
\end{abstract}

\section{Introduction}

We consider the problem
\begin{equation}
\label{main_prob}
\min_{x\in X}\; f(x)
\end{equation}
where $X\subset \Rb^n$ is convex and closed, and $f:\Rb^n\to\Rb$ is a Lipschitz continuous function, which may be neither convex nor smooth. The subgradients of $f(\cdot)$ are not available; instead, we postulate access to their random estimates.

Research on stochastic subgradient methods for nonsmooth and nonconvex functions started in late 1970's. Early contributions
are due to Nurminski, who considered weakly convex functions and established a general methodology for studying convergence of non-monotonic methods \cite{Nurminski1979numerical},
Gupal and his co-authors, who considered convolution smoothing (mollification) of Lipschitz functions
and resulting finite-difference methods \cite{gupal1979stochastic}, and  Norkin,
who considered unconstrained problems with ``generalized differentiable'' functions
\cite[Ch. 3 and 7]{mikhalevich1987nonconvex}.
 Recently, by an approach via differential inclusions, Duchi and Ruan \cite{duchi2018stochastic} studied proximal methods for
 sum-composite problems with weakly convex functions,
 Davis \emph{et al.} \cite{davis2018stochastic} proved convergence of the subgradient method for locally Lipschitz Whitney stratifiable functions
 with constraints, and Majewski \emph{et al.} \cite{majewski2018analysis} studied several methods for subdifferentially regular Lipschitz functions.

Our objective is to show that a single time-scale stochastic subgradient method with direction averaging \cite{ruszczynski1986method,ruszczynski87},
is convergent for a broad class of functions
enjoying the property of ``generalized differentiability,'' which contains all classes of functions mentioned above, as well as their compositions.

Our analysis follows the approach of relating a stochastic approximation algorithm to a
continuous-time dynamical system, pioneered in \cite{ljung1977analysis,kushner2012stochastic} and developed in
many works (see, \emph{e.g.}, \cite{kushner2003stochastic} and the references therein). Extension to multifunctions
was proposed in  \cite{benaim2005stochastic} and further developed, among others,
in \cite{borkar2009stochastic,duchi2018stochastic,davis2018stochastic,majewski2018analysis}.

For the purpose of our analysis, we also prove a chain rule on a path under generalized differentiability, which may be of independent interest.
%Our analysis reveals a theoretical advantage of path-averaging of stochastic subgradients: the primal component of the resulting path is smoother and our chain rule applies to it. The ``rough'' dual part of the path enters a smooth part of the Lyapunov function only.

{ We illustrate the use of the method for training a \emph{ReLu} neural network.}

\section{The chain formula on a path}

Norkin \cite{norkin1980generalized} introduced the following class of functions.
\begin{definition}
\label{d:Norkin}
A function $f:\Rb^n\to\Rb$ is  \emph{differentiable in a generalized sense at a point} $x\in \Rb^n$,
if an open set $U\subset \Rb^n$ containing $x$,
and a nonempty, convex, compact valued, and upper semicontinuous multifunction
$G_f: U \rightrightarrows \Rb^n$ exist, such that for all $y\in U$ and all $g \in G_f(y)$ the following equation is true:
\[
f(y) = f(x) + \langle g(y), y-x \rangle + o(x,y,g),
\]
with
\[
\lim_{y\to x} \sup_{g\in G(y)} \frac{o(x,y,g)}{\|y-x\|}=0.
\]
The set $G_f(y)$ is  the \emph{generalized subdifferential} of $f$ at $y$.
If a function is differentiable in a generalized sense at every  $x \in \Rb^n$ with the same generalized subdifferential mapping $G_f:\Rb^n\rightrightarrows \Rb^n$, we call it \emph{differentiable in a generalized sense}.

{ A function $f:\Rb^n\to\Rb^m$ is  differentiable in a generalized sense, if each of its component functions, $f_i:\Rb^n\to\Rb$, $i=1,\dots,m$, has this property.}
\end{definition}

The class of such functions is contained in the set of locally Lipschitz functions \cite[Thm. 1.1]{mikhalevich1987nonconvex},
and contains all subdifferentially regular functions \cite{clarke1975generalized},
Whitney stratifiable Lipschitz functions \cite{drusvyatskiy2015curves},
semismooth functions \cite{mifflin1977semismooth}, and their compositions.
In fact,
if a function is differentiable in  generalized sense and  has directional derivatives at $x$ in every direction, then it is semismooth at $x$. The Clarke subdifferential $\partial\! f(x)$ is an inclusion-minimal generalized subdifferential, but the generalized sub\-differential
mapping $G_f(\cdot)$ is not uniquely defined in Definition \ref{d:Norkin}, which plays a role in our considerations. For stochastic optimization, essential is the closure of the class of such functions  with respect to
expectation, which allows for easy generation of stochastic subgradients.
In the Appendix we recall basic properties of functions  differentiable in a generalized sense.
For thorough exposition, see \cite[Ch. 1 and 6]{mikhalevich1987nonconvex}.

Our interest is in a formula for calculating the increment of a function $f:\Rb^n\to \Rb$ along a path $p:[0,\infty)\to\Rb^n$,
which is at the core of
the analysis of nonsmooth and stochastic optimization algorithms (see \cite{drusvyatskiy2015curves,davis2019stochastic}
and the references therein).
 For an absolutely continuous function $p:[0,\infty)\to\Rb^n$ we denote
by $\dt{p}(\cdot)$ its weak derivative, that is, a measurable function such that
\[
p(t) = p(0) + \int_0^t \dt{p}(s)\;ds,\quad \forall \;  t \ge 0.
\]
%The following theorem extends the chain rule to functions differentiable in a generalized sense.
\begin{theorem}
\label{t:chain-path}
If\, $f:\Rb^n\to\Rb$ and $p:[0,\infty)\to \Rb^n$  are differentiable in a generalized sense, then for every $T>0$,
any generalized subdifferential $G_f(\cdot)$, and every selection $g(p(t))\in G_f(p(t))$, we have
\begin{equation}
\label{chain-path}
f(p(T))- f(p(0)) = \int_0^T \big\langle g(p(t)), \dt{p}(t) \big\rangle \;dt.
\end{equation}
\end{theorem}
\begin{proof}
Consider the function
\[
\varphi(\varepsilon) = \int_0^T f(p(t+\varepsilon))\;dt,\quad \varepsilon\ge 0.
\]
Its right derivative at 0 can be calculated in two ways:
\begin{equation}
\label{varphi-prime}
\begin{aligned}
\varphi'_+(0) &= \lim_{\varepsilon\downarrow 0} \frac{1}{\varepsilon} \big[ \varphi(\varepsilon) - \varphi(0)\big] =
\lim_{\varepsilon\downarrow 0} \frac{1}{\varepsilon} \Big[ \int_0^T f(p(t+\varepsilon))\;dt - \int_0^T f(p(t))\;dt \Big]\\
&= \lim_{\varepsilon\downarrow 0} \frac{1}{\varepsilon} \Big[ \int_\varepsilon^{T+\varepsilon} f(p(\tau))\;d\tau - \int_0^T f(p(t))\;dt \Big]\\
&= \lim_{\varepsilon\downarrow 0} \frac{1}{\varepsilon} \Big[ \int_T^{T+\varepsilon} f(p(t))\;dt - \int_0^\varepsilon f(p(t))\;dt \Big]
= f(p(T))- f(p(0)).
\end{aligned}
\end{equation}
On the other hand,
\begin{equation}
\label{pre-Lebesgue}
\varphi'_+(0) = \lim_{\varepsilon\downarrow 0}  \int_0^T \frac{1}{\varepsilon} \big[f(p(t+\varepsilon)) -f(p(t))\big]\;dt.
\end{equation}
By the generalized differentiability of $f(\cdot)$, the differential quotient under the integral can be expanded as follows:
\begin{align}
\lefteqn{\frac{1}{\varepsilon} \big[f(p(t+\varepsilon))\;dt -f(p(t))\big]} \nonumber \\
 &=\frac{1}{\varepsilon} \big\langle g(p(t+\varepsilon)), p(t+\varepsilon)-p(t) \big\rangle
 + \frac{1}{\varepsilon} o\big(p(t), p(t+\varepsilon), g(p(t+\varepsilon)) \big), \label{expansion}\\
&\text{with}\quad \lim_{\varepsilon\downarrow 0} \frac{1}{\varepsilon} o\big(p(t), p(t+\varepsilon), g(p(t+\varepsilon)) \big) =0.
\nonumber
\end{align}
{ Since $p(\cdot)$ is differetiable in a generalized sense, it is locally Lipschitz continuous
\cite[Thm. 1.1]{mikhalevich1987nonconvex}, hence absolutely continuous.} Thus,
for almost all $t$, we have
$\frac{1}{\varepsilon}\big[ p(t+\varepsilon)-p(t)\big]  = \dt{p}(t) + r(t,\varepsilon)$,
with $\lim_{\varepsilon\downarrow 0}\; r(t,\varepsilon)=0$.
Combining  it with \eqref{expansion}, and using the local boundedness of generalized gradients, we obtain
\begin{equation}
\label{quotient}
\frac{1}{\varepsilon}\big[f(p(t+\varepsilon)) -f(p(t))\big]
=
\big\langle g(p(t+\varepsilon)), \dt{p}(t) \big\rangle + O(t,\varepsilon),
\end{equation}
with $ \lim_{\varepsilon\downarrow 0}\;O(t,\varepsilon)=0$.
By \cite[Thm. 1.6]{mikhalevich1987nonconvex} (Theorem \ref{t:composition}), the function $\psi(t) = f(p(t))$ is differentiable in a generalized sense as well and
\[
G_\psi(t) = \big\{ \langle g, h \rangle: g\in G_f(p(t)), \; h\in G_p(t) \big\}
\]
is its generalized subdifferential. By virtue of \cite[Cor. 1.5]{mikhalevich1987nonconvex} (Theorem \ref{t:one-var}), any generalized subdifferential mapping $G_\psi(\cdot)$ is single-valued
except for a countable number of points in $[0,1]$. Since it is upper semicontinuous, it is continuous almost everywhere.
By \cite[Thm. 1.12]{mikhalevich1987nonconvex} (Theorem \ref{t:differetiable-ae}), almost everywhere $G_p(t)=\{\dt{p}(t)\}$. Then for any $h(t+\varepsilon)\in G_p(t+\varepsilon)$
and for almost all~$t$,
\[
\lim_{\varepsilon\downarrow 0} \big\langle g(p(t+\varepsilon)), h(t+\varepsilon) \big\rangle = \big\langle g(p(t)), \dt{p}(t) \big\rangle.
\]
Therefore, for almost all $t$,
\[
\lim_{\varepsilon\downarrow 0} \big\langle g(p(t+\varepsilon)), \dt{p}(t) \big\rangle 
=
\big\langle g(p(t)), \dt{p}(t) \big\rangle +
\lim_{\varepsilon\downarrow 0} \big\langle g(p(t+\varepsilon)), \dt{p}(t) - h(t+\varepsilon) \big\rangle
= \big\langle g(p(t)), \dt{p}(t) \big\rangle,
\]
where the last equation follows from the local boundedness of $G_f(\cdot)$  and the continuity of $G_p(\cdot)$ at the points of differentability. Thus,
for almost all $t$, we can pass to the limit in \eqref{quotient}:
\[
\lim_{\varepsilon\downarrow 0} \frac{1}{\varepsilon} \big[f(p(t+\varepsilon))\;dt -f(p(t))\big] = \big\langle g(p(t)), \dt{p}(t) \big\rangle.
\]
We can now use the Lebesgue theorem and pass to the limit under the integral in \eqref{pre-Lebesgue}:
\[
\varphi'_+(0)  =\int_0^T \lim_{\varepsilon\downarrow 0}   \frac{1}{\varepsilon} \big[f(p(t+\varepsilon))\;dt -f(p(t))\big]\;dt
= \int_0^T \big\langle g(p(t)), \dt{p}(t) \big\rangle \;dt.
\]
Comparison with \eqref{varphi-prime} yields \eqref{chain-path}. \qed
\end{proof}

\section{The single time-scale method with subgradient averaging}

We briefly recall from \cite{ruszczynski1986method,ruszczynski87} a stochastic approximation algorithm for solving problem \eqref{main_prob} where only random estimates  of subgradients of $f$ are available.

The method generates two random sequences: approximate solutions $\{x^k\}$ and
path-averaged stochastic subgradients $\{z^k\}$, defined
on a certain probability space $(\Omega,\Fc,P)$. We let $\Fc_k$ to be
the $\sigma$-algebra generated by
$\{x^0,\dots,x^k,z^0,\dots,z^k\}$.
We assume that
for each $k$, we can observe an $\Fc_k$-measurable random vector \mbox{$g^{k} \in \Rb^n$}, such that,
for some $\Fc_{k}$-measurable vector $r^{k}$, we have
$g^{k} - r^{k} \in G_f(x^k)$.
Further assumptions on the errors $r^k$ will be specified in section \ref{s:analysis}.

The method proceeds for $k=0,1,2\dots$ as follows ($a>0$ and $\beta>0$ are fixed parameters).
We compute
\begin{equation}
\label{QP}
y^k = \argmin_{y \in X}\  \left\{\langle z^k, y-x^k \rangle + \frac{\beta}{2} \|y-x^k\|^2\right\},
\end{equation}
and, with an $\Fc_k$-measurable stepsize $\tau_k \in \big(0,\min(1,1/a)\big]$, we set
\begin{equation}
\label{def_xk}
x^{k+1} = x^k + \tau_k (y^k-x^k).
\end{equation}
Then we observe $g^{k+1}$ at $x^{k+1}$, and update the averaged stochastic subgradient as
\begin{equation}
z^{k+1} = (1-a\tau_k)z^k + a\tau_k g^{k+1}.
\label{def_zk}
\end{equation}
Convergence of the method was proved in \cite{ruszczynski87} for weakly convex functions $f(\cdot)$.  Unfortunately, this class does
not contain functions with
downward cusps, which are common in modern machine learning models (see section \ref{s:example}).

\section{Convergence analysis}
\label{s:analysis}

We call a point $x^*\in \Rb^n$ \emph{Clarke stationary} of problem \eqref{main_prob}, if
\begin{equation}
\label{stationary}
0 \in \partial\! f(x^*) + N_X(x^*),
\end{equation}
where $N_X(x^*)$ denotes the normal cone to $X$ at $x^*$. The set of Clarke stationary points of problem \eqref{main_prob}
is denoted by $X^*$.

We start from a useful property of the gap function \mbox{$\eta:X\times \Rb^n\to(-\infty,0]$},
\begin{equation}
\label{gap}
\eta(x,z) = \min_{y\in X} \left\{\langle z, y-x \rangle + \frac{\beta}{2} \|y-x\|^2\right\}.
\end{equation}
We denote the minimizer in \eqref{gap} by $\bar{y}(x,z)$. Since it is a projection of $x-z/\beta$ on $X$, we observe that
\begin{equation}
\label{opt-eta}
 \langle z ,\bar{y}(x,z)-x \rangle  + \beta \| \bar{y}(x,z)-x \|^2 \le 0.
\end{equation}
Moreover, a point $x^*\in X^*$ if and only if $g^*\in \partial\! f(x^*)$ exists such that $\eta(x^*,g^*)=0$.

We analyze convergence of the algorithm \eqref{QP}--\eqref{def_zk} under the following conditions,
the first three of which are assumed to hold with probability 1:
\begin{description}
\item[(A1)] All iterates $x^k$ belong to a compact set;
\item[(A2)] $\tau_k \in \big(0,\min(1,1/a)\big]$ for all $k$, $\lim_{k\to\infty}\tau_k= 0$, $\sum_{k=0}^\infty \tau_k = \infty$;
\item[(A3)] For all $k$, $r^k = e^k+\delta^k$, with $\sum_{k=0}^\infty \tau_k e^k$ convergent, and $\lim_{k\to\infty}\delta_k = 0$;
\item[(A4)] The set $\{ f(x): x\in X^*\}$ does not contain an interval of nonzero length.
\end{description}
Condition (A3) can be satisfied for a martingale $\sum_{k=0}^\infty \tau_k e^k$, but can also hold for broad classes of dependent
``noise'' sequences $ \{e^k\}$ \cite{kushner2003stochastic}. { Condition (A4) is true for Whitney stratifiable functions
\cite[Cor. 5]{bolte2007clarke}, but we need to assume it here.}

We  have the following elementary property of the sequence $\{z^k\}$.
\begin{lemma}
\label{l:zk-bounded}
Suppose the sequence $\{x^k\}$ is included in a set $A\subset \Rb^n$ and conditions {\rm (A2)} and {\rm (A3)} are satisfied. Then
\[
\lim_{k\to\infty} \text{\rm dist}(z^k,B)=0,\quad {where}\quad B = \text{\rm conv} \Big(\bigcup_{x\in A} \partial\! f(x)\Big).
\]
\end{lemma}
\begin{proof} Using (A2), we define the quantities $\tilde{z}^k = z^k + a \sum_{j=k}^\infty \tau_j e^j$
and establish the recursive relation
\[
\tilde{z}^{k+1} = (1-a\tau_k) \tilde{z}^k + a\tau_k g^k + \tau_k \Delta_k, \quad k=0,1,2,\dots,
\]
where $g^k\in B$ and $\Delta_k = a\delta^k + a \sum_{j=k}^\infty \tau_j e^j\to 0$ a.s.. The convexity of the distance function and (A2) yield the result. \qed
\end{proof}

\begin{theorem}
\label{t:convergence}
If assumptions {\rm (A1)--(A4)} are satisfied,
then, with probability 1, every accumulation point $\hat{x}$ of the sequence $\{x^k\}$ is Clarke stationary,
 and the sequence $\{f(x^k)\}$ is convergent.
\end{theorem}
\begin{proof}

Due to (A1), by virtue of Lemma \ref{l:zk-bounded}, the sequence $\{z^k\}$ is bounded. We divide the proof into three standard steps.

\emph{Step 1: The Limiting Dynamical System.}  We define $p^k=(x^k,z^k)$, accumulated stepsizes
$t_k = \sum_{j=0}^{k-1}\tau_j$, $k=0,1,2 \dots$, and we construct the interpolated trajectory
\[
P_0(t) = p^k + \frac{t-t_{k}}{\tau_k}(p^{k+1}-p^k),\quad t_{k}\le t \le t_{k+1},\quad k=0,1,2,\dots.
\]
For an increasing sequence of positive numbers $\{s_k\}$ diverging to $\infty$, we define shifted trajectories $P_k(t) =P_0(t+s_k)$.
Recall that $P_k(t) = \big( X_k(t),Z_k(t)\big)$.

By \cite[Thm. 3.2]{majewski2018analysis},  for any increasing sequence $\{n_k\}$ of positive integers,
there exist a subsequence $\{ \tilde{n}_k\}$ and  absolutely continuous functions $X_\infty:
[0,+\infty) \to X$ and $Z_\infty: [0,+\infty) \to \Rb^n$ such that for any $T > 0$
\[
\lim_{k\to\infty} \sup_{t\in[0,T]}\left(\big\| X_{\tilde{n}_k}(t)-X_\infty(t)\big\| + \big\| Z_{\tilde{n}_k}(t)-Z_\infty(t)\big\|\right)= 0,
\]
and $(X_\infty(\cdot),Z_\infty(\cdot))$ is a solution of the system of differential equations and inclusions:
\begin{align}
\dt{x}(t) &= \bar{y}\big(x(t),z(t)\big)-x(t), \label{dx}\\
\dt{z}(t) &\in a \big(\partial\! f(x(t)) - z(t)\big). \label{dz}
\end{align}
Moreover, for any $t\ge 0$, the pair $(X_\infty(t),Z_\infty(t))$ is an accumulation point of the sequence $\{(x^k,z^k)\}$.

\emph{Step 2: Descent Along a Path.}  We use the Lyapunov function
\[
W(x,z) = af(x) - \eta(x,z).
\]
For any solution $(X(t),Z(t))$ of the system \eqref{dx}--\eqref{dz}, and for any $T>0$,
we estimate the difference $W(X(T),Z(T)) - W(X(0),Z(0))$. {
We split $W(X(\cdot),Z(\cdot))$ into a generalized differentiable composition $f(X(\cdot))$ and
the ``classical'' part $\eta(X(\cdot),Z(\cdot))$.

Since the path $X(\cdot)$ satisfies \eqref{dx} and $\bar{y}(\cdot,\cdot)$ is continuous, $X(\cdot)$ is continuously differentiable.} Thus,
we can use Theorem \ref{t:chain-path} to conclude that for any $g(X(\cdot))\in \partial\! f(X(\cdot))$
\begin{equation}
\label{f-incr}
\qquad f(X(T)) - f(X(0)) = \int_0^T \big\langle g(X(t)), \dt{X}(t) \big\rangle \;dt 
= \int_0^T \big\langle g(X(t)), \bar{y}(X(t),Z(t)) - X(t) \big\rangle \;dt.\qquad
\end{equation}
On the other hand, since $\bar{y}(x,z)$ is unique, the function $\eta(\cdot,\cdot)$ is continuously differentiable.
Therefore, the chain formula holds for it as well:
\[
\eta(X(T),Z(T)) - \eta(X(0),Z(0)) 
= \int_0^T  \big\langle \nabla_x \eta(X(t),Z(t)), \dt{X}(t) \big\rangle \;dt +
\int_0^T  \big\langle \nabla_z \eta(X(t),Z(t)), \dt{Z}(t) \big\rangle \;dt.
\]
Substituting $\nabla_x \eta(x,z) = -z+\beta(x-\bar y(x,z))$, $\nabla_z \eta(x,z) = \bar y(x,z)-x$ and
$\dt{Z}(t) =  a\big ( \hat{g}(X(t)) - Z(t)\big)$ with some $\hat{g}(X(\cdot))\in \partial\! f(X(\cdot))$,
and using \eqref{opt-eta} we obtain
\begin{align*}
\lefteqn{\eta(X(T),Z(T)) - \eta(X(0),Z(0))}  \\
&= \int_0^T  \big\langle -Z(t)+\beta(X(t)-\bar y(X(t),Z(t)))\, , \,\bar{y}(X(t),Z(t)) - X(t) \big\rangle \;dt \\
&{\quad } +
a \int_0^T  \big\langle \bar y(X(t),Z(t))-X(t)\,,\, \hat{g} (X(t)) - Z(t) \big\rangle \;dt\\
&\ge\; a \int_0^T  \big\langle \bar y(X(t),Z(t))-X(t)\,,\, \hat{g} (X(t)) - Z(t) \big\rangle \;dt\\
&\ge\; a \int_0^T  \big\langle \bar y(X(t),Z(t))-X(t)\,,\, \hat{g} (X(t))\big\rangle \;dt
 + a\beta \int_0^T  \big\| \bar y(X(t),Z(t))-X(t)\big\|^2 \;dt.
\end{align*}
We substitute the subgradient selector $g(X(t))=\hat{g}(X(t))$ into \eqref{f-incr}  and combine it with the last inequality, concluding that
\begin{equation}
\label{W-incr}
W(X(T),Z(T)) - W(X(0),Z(0)) \\
\le  -a\beta \int_0^T  \big\| \bar y(X(t),Z(t))-X(t)\big\|^2 \;dt
= -a\beta \int_0^T  \big\| \dt X(t)\big\|^2 \;dt.
\end{equation}
\emph{Step 3: Analysis of Limit Points.} Define the set
$
\Sc = \big\{ (x,z)\in X^* \times \Rb^n: \eta(x,z)=0\big\}$.
Suppose $(\bar{x},\bar{z})$ is an accumulation point of the sequence $\{(x^k,z^k)\}$. If $\eta(\bar{x},\bar{z}) <0$, then
every solution $(X(t),Z(t))$ of the system \eqref{dx}--\eqref{dz}, starting from $(X(0),Z(0))= (\bar{x},\bar{z})$ has $\dt X(0)\ne 0$.
Using \eqref{W-incr} and arguing as in \cite[Thm. 3.20]{duchi2018stochastic} or \cite[Thm. 3.5]{majewski2018analysis},
we obtain a contradiction. Therefore,
we must have $\eta(\bar{x},\bar{z})=0$. Suppose $\bar{x}\not \in X^*$. Then
\begin{equation}
\label{non-opt}
\dist\big(0,  \partial\! f(\bar{x}) + N_X(\bar{x})\big) >0.
\end{equation}
Suppose $X(t)=\bar{x}$ for all $t\ge 0$. The inclusion \eqref{dz} simplifies:
$\dt{z}(t) \in a \big(\partial\! f(\bar{x}) - z(t)\big)$.
By using the convex Lyapunov function $V(z) = \dist\big(z,\partial\! f(\bar{x})\big)$ and applying the classical chain formula
on the path $Z(\cdot)$ \cite{brezis1971monotonicity}, we deduce that
\begin{equation}
\label{Z-conv}
\lim_{t\to\infty} \dist\big(Z(t),  \partial\! f(\bar{x}) \big) = 0.
\end{equation}
It follows from \eqref{non-opt}--\eqref{Z-conv} that $T>0$ exists,
such that $ -Z(T) \not \in N_X(\bar{x})$, which yields $\dt X(T)\ne 0$. Consequently,
the path $X(t)$ starting from $\bar{x}$ cannot be constant.  But then again $T>0$ exists, such that $\dt X(T)\ne 0$. By Step 1,
the pair $(X(T),Z(T))$ would have to be an accumulation point of of the sequence $\{(x^k,z^k)\}$, a case already excluded.
We conclude that every accumulation point $(\bar{x},\bar{z})$  of the sequence $\{(x^k,z^k)\}$ is in $\Sc$.
The convergence of the sequence $\big\{W(x^k,z^k)\big\}$ then follows in the same way as
\cite[Thm. 3.20]{duchi2018stochastic} or \cite[Thm. 3.5]{majewski2018analysis}. As $\eta(x^k,z^k)\to 0$,
the convergence of $\{f(x^k)\}$ follows as well.  \qed
\end{proof}
Directly from Lemma \ref{l:zk-bounded} we obtain convergence of averaged stochastic subgradients.
\begin{corollary}
If the sequence $\{x^k\}$ is convergent to a single point $\bar{x}$, then every accumulation point
of $\{z^k\}$ is an element of $\partial\! f(\bar{x})$.
\end{corollary}

{
\section{Example}
\label{s:example}

A \emph{Rectified Linear Unit (ReLU)} neural network \cite{nair2010rectified} predicts a random outcome $Y\in \Rb^m$ from random features $X\in \Rb^n$ by a nonconvex nonsmooth function $y(X,W)$, defined recursively as follows:
\[
s_1=X,\quad s_{\ell+1} = (W_\ell s_\ell)_+,\  \ell = 1,2,\dots,L-1,\quad y(X,W)=W_L s_L,
\]
where $(v)_+=\max(0,v)$, componentwise.
The decision variables are  $W_1,\dots,W_{L-1}\in \Rb^{n\times n}$ and $W_L\in \Rb^{m\times n}$. The simplest training problem is:
\begin{equation}
\label{ReLu}
\min_{W\in \mathcal{W}}\;f (W) \eqdef \frac{1}{2} \Eb \big[ \|y(X,W)-Y\|^2\big],
\end{equation}
where $\mathcal{W}$ is a box about 0.
%Various versions of this problem exist, with sparsity inducing regularization, layer-dependent dimension, etc.
The function $f(W)$ is not subdifferentially regular. It is not Whitney stratifiable, in general, because this property is not preserved
under the expected value operator. However, we can use Theorems \ref{t:composition} and \ref{t:expected-generalized} to verify that it is differentiable in a generalized sense, and to calculate its stochastic subgradients. For a random data point $(X^k,Y^k)$ we subdifferentiate the function under the expected value in \eqref{ReLu} by recursive application of Theorem \ref{t:composition}.
In particular, for $L=2$ and $m=1$ we have $y(X,W) =  W_2 (W_1X)_+ $, and
$g^k = \big( y(X^k,W^k)-Y^k \big)
\begin{bmatrix}
\; D^k (W_2^k)^T  (X^k)^T
\,& \,(W_1^kX^k)^T_+\;
\end{bmatrix}.
$
He\-re, $D^k$ is a diagonal matrix with 1 on position $(i,i)$, if $(W_1^k X^k)_i>0$, and 0 otherwise. A typical run of the stochastic subgradient method and the method with direction averaging is shown in Fig. \ref{f:1},
on an example of predicting wine quality \cite{cortez2009modeling}, with identical random starting points, sequences of observations, and schedules of stepsizes: $\tau_k = 0.03/(1+5 k/N)$, where $N = 500,000$.
The coefficient $a=0.1$. For comparison, the loss of a simple regression model is 666.}
\begin{figure}[h!]
\vspace{-17em}
\centering
\includegraphics[width=\linewidth]{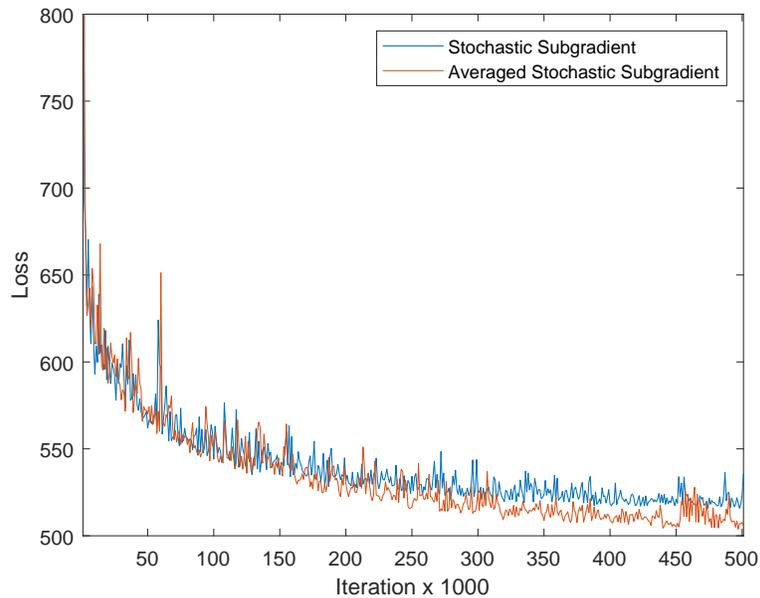}
\vspace{-17em}
\caption{Comparison of  methods with (lower graph) and without averaging (upper graph).}\label{f:1}
\end{figure}

\bibliographystyle{plain}

\begin{thebibliography}{10}

\bibitem{benaim2005stochastic}
M.~Bena{\"\i}m, J.~Hofbauer, and S.~Sorin.
\newblock Stochastic approximations and differential inclusions.
\newblock {\em SIAM Journal on Control and Optimization}, 44(1):328--348, 2005.

\bibitem{bolte2007clarke}
J.~Bolte, A.~Daniilidis, A.~Lewis, and M.~Shiota.
\newblock Clarke subgradients of stratifiable functions.
\newblock {\em SIAM Journal on Optimization}, 18(2):556--572, 2007.

\bibitem{borkar2009stochastic}
V.~S. Borkar.
\newblock {\em Stochastic Approximation: a Dynamical Systems Viewpoint}.
\newblock Springer, New York, 2009.

\bibitem{brezis1971monotonicity}
H.~Br{\'e}zis.
\newblock Monotonicity methods in {H}ilbert spaces and some applications to
  nonlinear partial differential equations.
\newblock In {\em Contributions to Nonlinear Functional Analysis}, pages
  101--156. Elsevier, 1971.

\bibitem{clarke1975generalized}
F.~H. Clarke.
\newblock Generalized gradients and applications.
\newblock {\em Transactions of the American Mathematical Society},
  205:247--262, 1975.

\bibitem{cortez2009modeling}
P.~Cortez, A.~Cerdeira, F.~Almeida, T.~Matos, and J.~Reis.
\newblock Modeling wine preferences by data mining from physicochemical
  properties.
\newblock {\em Decision Support Systems}, 47(4):547--553, 2009.

\bibitem{davis2019stochastic}
D.~Davis and D.~Drusvyatskiy.
\newblock Stochastic model-based minimization of weakly convex functions.
\newblock {\em SIAM Journal on Optimization}, 29(1):207--239, 2019.

\bibitem{davis2018stochastic}
D.~Davis, D.~Drusvyatskiy, S.~Kakade, and J.~D. Lee.
\newblock Stochastic subgradient method converges on tame functions.
\newblock {\em Foundations of Computational Mathematics}, pages 1--36, 2018.

\bibitem{drusvyatskiy2015curves}
D.~Drusvyatskiy, A.~D. Ioffe, and A.~S. Lewis.
\newblock Curves of descent.
\newblock {\em SIAM Journal on Control and Optimization}, 53(1):114--138, 2015.

\bibitem{duchi2018stochastic}
J.~C. Duchi and F.~Ruan.
\newblock Stochastic methods for composite and weakly convex optimization
  problems.
\newblock {\em SIAM Journal on Optimization}, 28(4):3229--3259, 2018.

\bibitem{gupal1979stochastic}
A.~M. Gupal.
\newblock {\em Stochastic Methods for Solving Nonsmooth Extremal Problems}.
\newblock Naukova Dumka, Kiev, 1979.

\bibitem{kushner2003stochastic}
H.~Kushner and G.~G. Yin.
\newblock {\em Stochastic Approximation Algorithms and Applications}.
\newblock Springer, New York, 2003.

\bibitem{kushner2012stochastic}
H.~J. Kushner and D.~S. Clark.
\newblock {\em Stochastic Approximation Methods for Constrained and
  Cnconstrained Systems}.
\newblock Springer, New York, 1978.

\bibitem{ljung1977analysis}
L.~Ljung.
\newblock Analysis of recursive stochastic algorithms.
\newblock {\em IEEE Transactions on Automatic Control}, 22(4):551--575, 1977.

\bibitem{majewski2018analysis}
S.~Majewski, B.~Miasojedow, and E.~Moulines.
\newblock Analysis of nonsmooth stochastic approximation: the differential
  inclusion approach.
\newblock {\em arXiv preprint arXiv:1805.01916}, 2018.

\bibitem{mifflin1977semismooth}
R.~Mifflin.
\newblock Semismooth and semiconvex functions in constrained optimization.
\newblock {\em SIAM Journal on Control and Optimization}, 15(6):959--972, 1977.

\bibitem{mikhalevich1987nonconvex}
V.~S. Mikhalevich, A.~M. Gupal, and V.~I. Norkin.
\newblock {\em Nonconvex Optimization Methods}.
\newblock Nauka, Moscow, 1987.

\bibitem{nair2010rectified}
V.~Nair and G.~E. Hinton.
\newblock Rectified linear units improve restricted boltzmann machines.
\newblock In {\em Proceedings of the 27th International Conference on Machine
  Learning (ICML-10)}, pages 807--814, 2010.

\bibitem{norkin1980generalized}
V.~I. Norkin.
\newblock Generalized-differentiable functions.
\newblock {\em Cybernetics and Systems Analysis}, 16(1):10--12, 1980.

\bibitem{Nurminski1979numerical}
E.~A. Nurminski.
\newblock {\em Numerical Methods for Solving Deterministic and Stochastic
  Minimax Problems}.
\newblock Naukova Dumka, Kiev, 1979.

\bibitem{ruszczynski1986method}
A.~Ruszczy{\'n}ski.
\newblock A method of feasible directions for solving nonsmooth stochastic
  programming problems.
\newblock In F.~Archetti, G.~Di~Pillo, and M.~Lucertini, editors, {\em
  Stochastic Programming}, pages 258--271. Springer, 1986.

\bibitem{ruszczynski87}
A.~Ruszczy{\'n}ski.
\newblock A linearization method for nonsmooth stochastic programming problems.
\newblock {\em Mathematics of Operations Research}, 12(1):32--49, 1987.

\end{thebibliography}

\newpage

\begin{appendices}
\section{Generalized differentiability of functions}

Compositions of generalized diifferentiable functions are crucial in our analysis.
\begin{theorem}{\rm \cite[Thm. 1.6]{mikhalevich1987nonconvex}}
\label{t:composition}
If $h:\Rb^m \to \Rb$ and $f_i:\Rb^n\to \Rb$, $i=1,\dots,m$, are differentiable in a generalized sense, then the composition
$\psi(x) = h\big( f_1(x),\dots,f_m(x)\big)$
is differentiable in a generalized sense, and at any point $x\in \Rb^n$ we can define the generalized subdifferential of $\psi$ as follows:
\begin{multline}
\label{sub-comp}
G_\psi(x) = \text{\rm conv} \big\{ g\in \Rb^n: g = \begin{bmatrix} g_1 & \cdots & g_m\end{bmatrix} g_0,\\ \text{ with } g_0\in G_{h}\big(f_1(x),\dots,f_m(x)\big) \text{ and } g_j\in G_{f_j}(x),\ j=1,\dots,m\big\}.
\end{multline}
\end{theorem}
Even if we take $G_{h}(\cdot)=\partial h(\cdot)$ and $G_{f_j}(\cdot)=\partial\! f_j(\cdot)$, $j=1,\dots,m$, we may obtain
$G_\psi(\cdot) \ne \partial \psi(\cdot)$,
but $G_\psi$ defined above satisfies Definition \ref{d:Norkin}. %We still have the following property of the Clarke subdifferential.

\begin{theorem}{\rm \cite[Thm. 1.12]{mikhalevich1987nonconvex}}
\label{t:differetiable-ae}
If $f:\Rb^n\to \Rb$ is differentiable in a generalized sense, then for almost all $x\in \Rb^n$ we have $G_f(x)=\{\nabla f(x)\}$.
\end{theorem}
Functions of one variable have the following remarkable property.
\begin{theorem}{\rm \cite[Cor. 1.5]{mikhalevich1987nonconvex}}
\label{t:one-var}
If $f:\Rb \to \Rb$ is differentiable in a generalized sense, then the set of points $x$ at which a generalized subdifferential $G_f(x)$ is not a singleton is at most countable.
\end{theorem}

For stochastic optimization, essential is the closure of the class functions differentiable in a generalized sense with respect to
expectation.
\begin{theorem}{\rm \cite[Thm. 23.1]{mikhalevich1987nonconvex}}
\label{t:expected-generalized}
Suppose $(\varOmega,\Fc,P)$ is a probability space and a function $f:\Rb^n\times \varOmega \to \Rb$ is differentiable in a generalized sense with respect to $x$ for all $\omega\in \varOmega$, and integrable with respect to $\omega$ for all $x\in \Rb^n$. Let $G_f: \Rb^n \times \varOmega
\rightrightarrows \Rb^n$ be a multifunction, which is measurable with respect to $\omega$ for all $x\in \Rb^n$, and which is a generalized subdifferential mapping of $f(\cdot,\omega)$ for all $\omega\in \varOmega$. If for every compact set $K\subset \Rb^n$ an integrable function
$L_K:\varOmega\to \Rb$ exists, such that
$\sup_{x\in K}\sup_{g\in G_f(x,\omega)}\|g\| \le L_K(\omega)$, $\omega \in \varOmega$,
then the function
\[
F(x) = \int_\varOmega f(x,\omega)\;P(d\omega),\quad x\in \Rb^n,
\]
is differentiable in a generalized sense, and the multifunction
\[
G_F(x) = \int_\varOmega G_f(x,\omega)\;P(d\omega),\quad x\in \Rb^n,
\]
is its generalized subdifferential mapping.
\end{theorem}
%This allows for easy construction of stochastic generalized subgradients by simulation:
%for a randomly selected elementary event $\widetilde{\omega}$,
% we can choose (in a measurable way) $\widetilde{g}\in G_f(x,\widetilde{\omega})$. Then, $E[\widetilde{g}]\in G_F(x)$.

\end{appendices}

\end{document}